\documentclass[11pt]{amsart}
\usepackage{amsmath}
\usepackage{amssymb}
\newtheorem{theorem}{Theorem}

\newtheorem{lemma}[theorem]{Lemma}

\newcommand{\C}{{\mathcal C}}

\newcommand{\HH}{{\mathcal H}}
\newcommand{\OO}{{\mathcal O}}
\newcommand{\D}{{\mathbb{D}}}
\newcommand{\G}{{\mathbb G}}
\newcommand{\Om}{\Omega}
\newcommand{\eps}{\varepsilon}
\title{Separate continuity of the Lempert function of the spectral
ball}

\author{Nikolai Nikolov, Pascal J. Thomas}

\address{Institute of Mathematics and Informatics\\ Bulgarian Academy
of Sciences\\ Acad. G. Bonchev 8, 1113 Sofia,
Bulgaria}\email{nik@math.bas.bg}

\address{Universit\'e de Toulouse\\ UPS, INSA, UT1, UTM \\
Institut de Math\'ematiques de Toulouse\\
F-31062 Toulouse, France} \email{pthomas@math.univ-toulouse.fr}



\begin{document}

\begin{thanks}{This note was written during the stay of the second
named author at the Institute of Mathematics and Informatics of
the Bulgarian Academy of Sciences supported by a CNRS grant
(Sempetber 2009).}
\end{thanks}

\begin{abstract} We find all matrices $A$ from the spectral unit ball
$\Omega_n$ such that the Lempert function $l_{\Omega_n}(A,\cdot)$
is continuous.
\end{abstract}

\maketitle

The characteristic polynomial of a $n\times n$ complex matrix $A$ is
$$
P_A(t):= \det(tI_n-A)=: t^n+\sum_{j=1}^n(-1)^j\sigma_j (A)t^{n-j},
$$
where $I_n$ is the unit matrix. Let
$r(A):=\max\{|\lambda|:P_A(\lambda)=0\}$ be the spectral radius of
$A$. The spectral unit ball is the pseudoconvex domain
$\Omega_n:=\{A:r(A)<1\}.$

Let
$\sigma (A):=(\sigma_1(A),\dots,\sigma_n(A))$.
The \emph{symmetrized polydisk} is the bounded domain
$\mathbb G_n:=\sigma(\Omega_n)\subset \mathbb C^n$,
which is hyperconvex (see \cite{Edi-Zwo}) and hence taut.

We are interested in two-point Nevanlinna--Pick problems with values
in the spectral unit ball, so let us consider
the Lempert function of a domain $D\subset\mathbb C^m$ : for $z,w\in D$,
$$
l_D(z,w):=\inf\{|\alpha|:\exists\varphi\in\mathcal O(\mathbb D,D):
\varphi(0)=z,\varphi(\alpha)=w\},
$$
where $\mathbb D\subset\mathbb C$ is the unit disc.
For general facts about this function, see for instance \cite{Jar-Pfl}. The Lempert
function is symmetric in its arguments, upper semicontinuous and decreases under
holomorphic maps, so for $A,B\in\Omega_n,$
\begin{equation}
\label{decr}
l_{\Omega_n}(A,B) \ge l_{\mathbb{G}_n}(\sigma(A),\sigma(B)).
\end{equation}
The domain $\mathbb G_n$ is taut, so its Lempert function is continuous.

The systematic study of the relationship between Nevanlinna--Pick
problems valued in the symmetrized polydisk or spectral ball began
with \cite{Agl-You}. In particular, it showed that when both $A$
and $B$ are \emph{cyclic} (or non-derogatory) matrices, i.e. they
admit a cyclic vector (see other equivalent properties in
\cite{NTZ}), then equality holds in \eqref{decr}. It follows that
$l_{\Omega_n}$ is continuous on $\mathcal C_n \times \mathcal
C_n$, where $\mathcal C_n$ denotes the (open) set of cyclic
matrices. On the other hand, in general, if equality holds in
\eqref{decr} at $(A,B)$, then $l_{\Omega_n}$ is continuous at
$(A,B)$ (see \cite[Proposition 1.2]{Tho-Tra}). The converse is
also true, since $l_{\Omega_n}$ is an upper semicontinuous
function, $l_{\Bbb G_n}$ is a continuous function and (\ref{decr}) holds.

The goal of this note is to study the continuity of $l_{\Omega_n}$
separately with respect to each argument. In \cite{Tho-Tra}, the authors
looked for matrices $B$ such that $l_{\Omega_n}(A,.)$ is continuous
at $B$ for any $A$. They conjecture that this holds for any $B \in \mathcal C_n$,
and prove it for $n\le 3$ \cite[Proposition 1.4]{Tho-Tra}, and the converse
statement for all dimensions (see \cite[Theorem 1.3]{Tho-Tra}).

In the present paper, we ask for which $A$ the function $l_{\Omega_n}(A,.)$
is continuous at $B$ for any $B$ (or simply, continuous on the whole $\Omega_n$).
By \cite[Proposition 4]{NTZ}, for any matrix $A\in\C_n$ with at least two
different eigenvalues, the function $l_{\Om_n}(A,\cdot)$ is not
continuous at any scalar matrix. On the other hand,
$l_{\Om_n}(0,B)=r(B)$ and hence $l_{\Om_n}(A,\cdot)$ is a continuous function for any scalar matrix $A$
(since the automorphism $\Phi_{\lambda}(X)=(X-\lambda I)(I-\overline{\lambda}X)^{-1}$
of $\Omega_n$ maps $\lambda I_n$ to $0$, where $\lambda\in\mathbb D$).

We have already mentioned that if $A\in\Om_n$ ($n\ge 2$), then the following conditions are equivalent:

(i) the function $l_{\Om_n}$ is continuous at $(A,B)$ for any $B\in\Om_n;$

(ii) $l_{\Omega_n}(A,\cdot)=l_{\mathbb{G}_n}(\sigma(A),\sigma(\cdot)).$

\noindent Consider also the condition

(iii) $A\in\mathcal C_2$ has two equal eigenvalues.

\noindent By \cite[Theorem 8]{Cos}, (iii) implies (ii). Theorem 1 below says that the
scalar matrices and the matrices satisfying (iii) are the only cases when $l_{\Om_n}(A,\cdot)$
is a continuous function. Then the mentioned above result \cite[Proposition 4]{NTZ} shows that
(i) implies (iii) and hence the conditions (i), (ii) and (iii) are equivalent.

\begin{theorem}\label{1} If $A\in\Om_n,$ then $l_{\Om_n}(A,\cdot)$
is a continuous function if and only if either $A$ is scalar or
$A\in\mathcal C_2$ has two equal eigenvalues.
\end{theorem}

\begin{proof} Using $\Phi_\lambda$ and an automorphisms of $\Omega_n$ of the form
$X\to P^{-1}XP,$ where $P$ is an invertible matrix, we may assume that $0$ is an
eigenvalue of $A$ and the matrix is in a Jordan form.

It is enough to prove that $l_{\Om_n}(A,\cdot)$ is not a
continuous function if $A$ has at least one non-zero eigenvalue or
$A\in\Om_n$ is a non-zero nilpotent matrix and $n\ge 3$.

In the first case, let $d_1\ge\dots\ge d_k$ be the numbers of the Jordan
blocks corresponding to the pairwise different eigenvalues
$\lambda_1=0,\lambda_2,\dots,\lambda_k.$ We
shall prove that $l_{\Om_n}(A,\cdot)$ is not continuous at $0.$ It
is easy to see that $A$ can be represented as blocks $A_1,\dots
A_l$ (with sizes $n_1,\dots,n_l$) such that the eigenvalues of
$A_1$ are equal to zero and the other blocks are cyclic with at
leat two different eigenvalues values ($A_1$ is missed if
$d_1=d_2$). By \cite[Proposition 4]{NTZ}, we know that there are $(A_{i,j})_j\to
0,$ $1\le i\le l,$ such that
$\sup_{i,j}l_{\Om_{n_i}}(A_i,A_{i,j}):=m<r(A).$ Taking $A_j$ to be
with blocks $A_{1,j},\dots,A_{l,j},$ it is easy to see
$l_{\Om_n}(A,A_j)\le\max_i l_{\Om_{n_i}}l(A_i,A_{i,j})\le
m<l_{\Om_n}(A,0)$ which implies that $l_{\Om_n}(A,\cdot)$ is not
continuous at $0.$

Let now $A\neq 0$ be a nilpotent matrix. Then
$A= (a_{ij})_{1\le i,j\le n}$ with $a_{ij}=0$
unless $j=i+1$. Let $r=\mbox{rank}(A)\ge 1$.
Following the proof
of Proposition 4.1 in \cite{Tho-Tra}, let
$$
F_0:=\{1\} \cup \left\{ j \in \{2,\dots,n\} : a_{j-1,j}=0 \right\}
:=\{ 1=b_1 < b_2 < \dots < b_{n-r} \},
$$
and $b_{n-r+1}:=n+1$. We set
 $d_i := 1+ \# \left( F_0 \cap \{(n-i+2),\dots,n\}\right) $.
The hypotheses on $A$ imply that we can choose its Jordan form
so that $a_{n-1,n}=1$, so
$1=d_1=d_2\le d_3\le\dots\le d_n= \#F_0 = n-r$, $d_{j+1}\le d_j+1$.

Corollary 4.3 and Proposition 4.1 in \cite{Tho-Tra} show
that for any $C\in\C_n$,
$$l_{\Om_n}(A,C)=h_{\G_n}(0,\sigma(C)):=\inf\{|\alpha|:\exists
\psi\in\HH(\D,\G_n):\psi(\alpha)=\sigma(C)\},$$ where
$$\HH(\D,\G_n)=\{\psi\in\OO(\D,\G_n):\mbox{ord}_0\psi_j\ge d_j,
\ 1\le j\le n\}.$$

Note that $d_j\le j-1$ for $j\ge 2.$ Let $m:=\min_{j\ge
2}\frac{d_j}{j-1}$ and choose a $k$ such that $\frac{d_k}{k-1}=m$.
If $m=1$, then ${d_j}= {j-1}$ for all $j\ge 2$,
and if furthermore $n\ge 3$, we can take $k=3$.

With $k$ chosen as above, let $\lambda$ be a small positive number,
$b=k\lambda^{k-1}$ and
$c=(k-1)\lambda^k.$ Then $\lambda$ is a double zero of the
polynomial $\Lambda(z) = z^{n-k}(z^k-bz+c)$ with zeros in $\mathbb D.$
Let $B$ be a diagonal matrix such that its characteristic polynomial is
$P_B(z)=\Lambda(z)$.

Assuming that $l_{\Om_n}(A,\cdot)$ is continuous at $B,$ then
$$l_{\Om_n}(A,B)=h_{\G_n}(0,\sigma(B))=:\alpha.$$

\begin{lemma}
\label{sumder}
If $l_{\Om_n}(A,B)=\alpha$, then there
is a $\psi\in\HH(\D,\G_n)$ with $\psi(\alpha)=\sigma(B)$ and
$$\sum_{j=1}^n\psi'_j(\alpha)(-\lambda)^{n-j}=0.$$
\end{lemma}

\begin{proof}
This is analogous to the proof of Proposition 4.1 in \cite{Tho-Tra}.
Let $\varphi \in \OO (\D, \Om_n)$ be such that $\varphi(0)=A$
and $\varphi(\tilde \alpha)=B$. Corollary 4.3  in \cite{Tho-Tra} applied to $A$
shows that $\tilde\psi := \sigma \circ \varphi \in \HH(\D,\G_n)$.

Now we study $\sigma_n(\varphi(\zeta))-\sigma_n(B)=\sigma_n(\varphi(\zeta))$
near $\zeta =\alpha$. We may assume that the first two diagonal coefficients of $B$
are equal to $\lambda$.  If we let $\varphi_\lambda (\zeta):=\varphi(\zeta)-\lambda I_n$,
then the first two columns of $\varphi_\lambda (\alpha)$ vanish, so
$\sigma_n \circ \varphi_\lambda = \det (\varphi_\lambda )$ vanishes to
order $2$ at $\alpha$.  On the other hand,
$$
\det (-\varphi_\lambda(\zeta) )=\det (\lambda I_n - \varphi(\zeta) )
= \lambda^n + \sum_{j=1}^n (-1)^j\lambda^{n-j} \tilde\psi_j(\zeta) ,
$$
and since the derivative of the left hand side vanishes at $\tilde\alpha$,
the same holds for the right hand side. It remains to let $\tilde\alpha\to\alpha$
and to use that $\Bbb G_n$ is a taut domain, providing the desired $\psi.$
\end{proof}

\begin{lemma}
\label{alphsm}
We have $\alpha^m\lesssim\lambda$; furthermore if $m=1$ and $n\ge 3$, then
$\alpha^{2/3}\lesssim\lambda$. So in all cases $\alpha\ll\lambda$.
\end{lemma}

\begin{proof}
Note that there is an $\eps>0$ such that for $\lambda<\eps$ the map
$\zeta\to(0,\dots,0,k(\eps\zeta)^{d_k},(k-1)\lambda(\eps\zeta)^{d_k},0,\dots,0)$
is a competitor for $h_{\Om_n}(A,B).$ So $(\eps\alpha)^{d_k}\le\lambda^{k-1}$,
that is, $\alpha^m\lesssim\lambda$.

If $m=1$ and $n\ge k=3,$ then considering the map
$\zeta\to(0,3\lambda^{1/2}\eps\zeta,2(\eps\zeta)^2,\\0\dots,0)$ we see that
$(\eps\alpha)^2\le\lambda^3$.
\end{proof}

Setting $\psi_j(\zeta)=\zeta^{d_j}\theta(\zeta),$
the condition in Lemma \ref{sumder} becomes
\begin{equation}
\label{condS}
a\frac{(-\lambda)^n}{\alpha}+S=0,
\end{equation}
where $a=(k-1)d_k-kd_{k-1}$ and
$S=\sum_{j=1}^n\alpha^{d_j}\theta'_j(\alpha)(-\lambda)^{n-j}.$
Note that $a\neq 0.$ Indeed, if $m<1,$ then $d_k=d_{k-1}$ and
hence $a=-d_k;$ if $m=1$, then $a=(k-1)(k-1) -k (k-2) = 1.$
Since $\mathbb G_n$ is bounded, $|\theta'_j(\alpha)|\lesssim 1.$

By Lemma \ref{alphsm} and the choice of $k$, for any $j$,
$$
\alpha^{d_j}\lesssim \lambda^{(k-1)d_j/d_k}  \le \lambda^{j-1}
\le\lambda^{n-1}.$$
Thus
$S\lesssim\lambda^{n-1}.$ By Lemma \ref{alphsm} again, $\alpha\ll\lambda,$
a contradiction with \eqref{condS}.
\end{proof}

\end{document}